\documentclass[12pt]{amsart}

\usepackage[usenames,dvipsnames]{xcolor}

\usepackage{amsmath,amsthm,amsfonts,graphicx,amssymb,amscd}
\usepackage[all,cmtip]{xy}
\usepackage{graphicx, overpic, float}
\usepackage[mathscr]{euscript}
\usepackage{hyperref}
\usepackage{pxfonts}

\usepackage{enumerate}

\newtheorem{thm}{Theorem}[section]
\newtheorem*{jthm}{Theorem}

\newtheorem{cor}[thm]{Corollary}
\newtheorem{lem}[thm]{Lemma}

\newtheorem{prop}[thm]{Proposition}

\theoremstyle{definition}
\newtheorem{defn}[thm]{Definition}

\newcommand{\ep}{\epsilon}
\newcommand{\Te}{Teichm\"{u}ller }

\usepackage{color}

\title[Limits of harmonic maps and crowned hyperbolic surfaces]{Limits of harmonic maps and\\ crowned hyperbolic surfaces}

\author{Subhojoy Gupta}

\address{Department of Mathematics, Indian Institute of Science, Bangalore 560012, India.}
\email{subhojoy@iisc.ac.in}

\begin{document}
\setcounter{tocdepth}{4}
\begin{abstract}  We consider harmonic diffeomorphisms to a fixed hyperbolic target $Y$, from a family of domain Riemann surfaces degenerating along a \Te ray.  We use the work of Minsky to show that there is a limiting harmonic map from the conformal limit of the \Te ray, to a crowned hyperbolic surface. The target surface is the metric completion of the complement of a  geodesic lamination on $Y$.  
 The conformal limit is obtained by attaching half-planes and cylinders to the critical graph of the holomorphic quadratic differential determining the ray. 
As an application, we provide a new proof of the existence of  harmonic maps from any punctured Riemann surface to  a given crowned hyperbolic target of the same topological type.\end{abstract}

\maketitle

\section{Introduction}
Let $g\geq 2$ and let $\mathcal{T}_g$ denote the Teichm\"{u}ller space, the space of marked Riemann surfaces of genus $g$. 
Our starting point is the following classical result:

\begin{jthm}[Eells-Sampson, Hartman, Al'ber, Sampson, Schoen-Yau] For any Riemann surface $X \in \mathcal{T}_g$  and for any diffeomorphism  $f:X\to Y$ where $Y$ is a hyperbolic surface, there is a  unique harmonic diffeomorphism $h$ that is homotopic to $f$.
\end{jthm}

The goal of this note is to understand the limiting behaviour of such harmonic maps when the domain surface diverges in $\mathcal{T}_g$.  The divergence we shall consider will be along a \Te ray, which is obtained by starting with a singular-flat metric induced by a  holomorphic quadratic differential $\Psi_0$, and stretching the vertical foliation $\Psi_0^h$ by a positive real parameter.

The resulting family of singular-flat surfaces (induced by quadratic differentials $\Psi_t$) has a \textit{conformal limit}  $X_\infty$, which can be thought of as the Gromov-Hausdorff limit when the basepoints are taken to be the set of singularities (or zeroes of the differential). This conformal limit, that we shall describe in \S3, is a  (possibly disconnected) punctured Riemann surface obtained by attaching Euclidean half-planes 
and half-infinite Euclidean cylinders  to the horizontal critical graph, of the initial quadratic differential $\Psi_0$. These infinite-area singular-flat metrics are called \textit{(generalized) half-plane structures}, as introduced in \cite{Gup25} (see also \cite{Gup2}).

On the hyperbolic geometry side, we shall interpret the limit of the harmonic map images as \textit{crowned hyperbolic surfaces}, which are non-compact surfaces with closed geodesic boundaries and  ``crown ends", each comprising a cyclic collection of bi-infinite geodesics bounding a finite number of boundary cusps (see \S2.4). A special example is an ideal hyperbolic polygon, which can be thought of as crowned hyperbolic surface of genus zero and a single crown end. \\

We shall prove:

\begin{thm}[Harmonic limits]\label{thm1}
Let $X$ be a Riemann surface of genus $g\geq  2$, and let $f:X\to Y$ be a diffeomorphism where $Y$ is a hyperbolic surface of genus $g$. Let  $\Psi_0$ be a holomorphic quadratic differential on $X$ with horizontal foliation $\Psi_0^h$. Let $\lambda$ be the geodesic lamination on $Y$ realizing $f(\Psi_0^h)$.

 Consider the family of surfaces  $\{X_t\}_{t>0}$ along the \Te ray determined by $\Psi_0$, and let $X_\infty$ be the conformal limit.

For each $t$, let   $h_t:X_t\to Y$ be the harmonic diffeomorphism homotopic to $f$.

Then  there is a subsequence $t_n\to \infty$ such that the maps $h_{t_n}$ converge to a harmonic map
\begin{equation}\label{hlim}
h:X_\infty \to \widehat{Y \setminus \lambda}
\end{equation}

where the target is the crowned hyperbolic surface that is the  metric completion of the complement of $\lambda$.

Conformally, each component of $X_\infty$ is a punctured Riemann surface, and the Hopf differential of $h$ has a pole of order greater than one at every puncture. 

\end{thm}

\textit{Remarks.} (i) The sense in which the maps above converge needs to be defined, since the domains are varying, and the energy of the maps are unbounded. This is done in Definition \ref{conv}, once the notion of ``conformal limit" of a family of Riemann surfaces is made precise.

  (ii) The work in \cite{WolfInf} essentially proves Theorem \ref{thm1} for Jenkin-Strebel differentials, that is, when all the components of the  lamination $\lambda$ are simple closed geodesics. (See, for example,  section 3 and Proposition 3.1 of that paper.) 
 
 (iii) In the generic case, the conformal limit of a \Te ray comprises $4g-4$ copies of complex plane (or the Riemann sphere with a single puncture) - see \S3. The lamination $\lambda$ is then maximal, and the metric completion of each of the $(4g-4)$ complementary regions is an ideal triangle. The limiting map is then a collection of  harmonic diffeomorphisms from $\mathbb{C}$ to an ideal hyperbolic triangle. Thus, the more interesting case for us would be when the lamination $\lambda$ is \textit{not} maximal, that is, when one of the complementary regions has positive genus.\\

The case of a degenerating domain was analyzed in the work of Minsky in \cite{Minsky}, and we shall rely on his work for proving Theorem \ref{thm1}.  Much of the analytic estimates were also in \cite{Wolf2} for the closely related case when the \textit{target} surface degenerates.  
The key technical tool that Minsky develops is a ``polygonal decomposition" of the domain surface equipped with the  singular-flat metric induced by the Hopf differential. The pieces  $\mathcal{P}_R$ of this decomposition are $R$-neighborhoods of the singularity-set, and the analytic estimates  (see Proposition \ref{gest}) provide control on the geometry of the map and its image in their complement $\mathcal{P}_R^c$ . Along the \Te ray $X_t$ ($t\to \infty$), we can choose increasing $R$, and the key component of the proof of Theorem \ref{thm1} is to show that after passing to a subsequence, the polygonal pieces converge on compacta to a singular flat surface that is conformally identical to the limit $X_\infty$.  For any fixed $R$, the corresponding harmonic maps have a uniform energy bound on the pieces of $\mathcal{P}_R$ along the ray, and standard methods then provide a further subsequence that converges to the limiting harmonic map $h$ defined on $X_\infty$. Moreover, we deduce from Minsky's work that for large $R$ the image of $\mathcal{P}_R^c$ is close to a lamination $\lambda$. It follows that in fact the image of $h$ is precisely the complement of the lamination $\lambda$ on the fixed hyperbolic target $Y$, the metric completion of which is a crowned hyperbolic surface. \\

As a corollary of Theorem \ref{thm1}, we provide a new proof of the following existence theorem for harmonic maps from a punctured Riemann surface to a crowned hyperbolic surface.   The existence of such harmonic maps was first proved in \cite{GupHar}, by a different method that involved taking a compact exhaustion of the domain.  

Existence results of harmonic maps with non-compact domains are difficult to prove, as the standard heat-flow or minimization techniques rely on the compactness of the domain. There are generalizations of the heat-flow method to the non-compact case (see for example  \cite{LiTam}) but these rely on the  \textit{finiteness} of energy or \textit{completeness} of the Riemannian target, both of which fail in our setting.  In the case when the domain is the complex plane, the existence of harmonic maps to ideal hyperbolic polygons was shown by \cite{HTTW} (see also \cite{Au-Tam-Wan}).
The first examples having domains of higher genus were given by A.Huang in his thesis (\cite{AHuang}), where he proved an existence result of harmonic maps from punctured Riemann surfaces to symmetric ideal polygons.  

Here, we shall prove:

\begin{thm}[Existence result] \label{thm2} Let $X$ be a Riemann surface of genus $g$ with a set of $k +l\geq 1$ marked points $P$, where $\chi(X\setminus P) <0$. Let $Z$ be a crowned hyperbolic surface of genus $g$ with $k$ crown ends and $l$ geodesic boundary components. Then there exists a harmonic diffeomorphism $h:X\setminus P \to Z$ whose Hopf differential has poles at the punctures, each of order greater than one.
\end{thm}

\textit{Remark.} The  harmonic map from a punctured Riemann surface to a crowned hyperbolic surface as in Theorem \ref{thm2} is not unique; different harmonic maps to the same crowned target has different rates of divergence into the boundary cusps of the crown ends. The non-uniqueness was  also characterized in previous work \cite{GupHar}, where we 
proved  an analogue of Wolf's parametrization of $\mathcal{T}_g$ (see \cite{Wolf0}), for the ``wild" Teichm\"{u}ller space of crowned hyperbolic surfaces. \\

The proof of Theorem \ref{thm2} is an immediate consequence of Theorem \ref{thm1}, once we can prescribe a \Te  ray and hyperbolic target in some suitable \Te space $\mathcal{T}_{g^\prime}$, that yield a given pair of limiting objects (namely, a punctured Riemann surface and a crowned hyperbolic surface).  This relies on two results that we prove in \S4 : first,  using the work in \cite{Gup2} that shows that any punctured Riemann surface can be equipped with a half-plane structure, we describe how it arises as a component of the conformal limit of some \Te ray (in a suitable $\mathcal{T}_g$). Second, an analogous result asserting that any crowned hyperbolic surface arises as the metric completion of the complement of a geodesic lamination on \textit{some} hyperbolic surface. The proofs use the technology of train-tracks, that we briefly review in \S2.3. \\

We remark that the theory of Higgs bundles offers an alternative perspective to our results; it is well known that solving the Hitchin equation is equivalent to determining an equivariant harmonic map from the universal cover of a Riemann surface  to a symmetric space, which in our (rank $2$) case is the hyperbolic plane (see \S11 of \cite{Hitchin}). Recently,  the work in \cite{Swob} studies the asymptotic behaviour of  such solutions of the Hitchin equations  (in the rank-$2$ setting), as the Riemann surface degenerates by pinching a curve, and limiting solutions  have been identified.  Such degenerations correspond to Teichm\"{u}ller rays determined by a Jenkins-Strebel differential, and the work in the present paper can interpreted as understanding  the limiting behaviour of the solutions of the Hitchin equations that lie in the Hitchin section, in the case of arbitrary \Te rays.\\

\medskip

\textbf{Acknowledgements.}  I am grateful to Mike Wolf for many conversations that led to the present work,  and to the International Center for Theoretical Sciences, Bangalore for hosting a meeting in November 2017, where he could visit.  I also thank the Infosys Foundation for their support.

\section{Preliminaries}

\subsection{Harmonic maps between surfaces}

We give a brief overview of some of the basic theory; for other accounts, see \cite{Jost-book}, \cite{Wolf0} or \cite{Wolf-Knoxville}.

In what follows, $(X,\sigma)$ and $(Y,\rho)$ will be two closed Riemann surfaces of the same genus, equipped with conformal metrics $\rho$ and $\sigma$. Here a \textit{conformal metric} $\rho$ is one that, in local coordinates, can be expressed as $\rho(z) \lvert dz^2 \rvert$ where $\rho(z)$ is a positive real-valued function (the \textit{conformal factor}). 

Throughout, we shall assume that $Y$ is hyperbolic, that is, the curvature of $\rho$ is $-1$ everywhere.

\begin{defn}The \textit{energy} of a diffeomorphism $w:(X,\sigma) \to (Y, \rho)$ is defined to be the integral of the $L^2$-norm of its gradient:
\begin{equation}
\mathcal{E}(w) = \displaystyle\int\limits_X e(z)  \sigma(z) dz d\bar{z}
\end{equation}
where the integrand $e$ is the \textit{energy density} defined as 
\begin{equation}
e(z) = \frac{\rho(w(z))}{\sigma(z)} \left( \lvert w_z\rvert^2 +  \lvert w_{\bar{z}}\rvert^2\right).
\end{equation}
Clearly, the energy of $w$ is independent of the choice of the conformal metric $\sigma$.
\end{defn}

\begin{defn} A harmonic map $w:(X,\sigma) \to (Y,\rho)$ is the \textit{least energy} map in its homotopy class. (Recall that here the metric $\rho$ is hyperbolic.)

The Euler-Lagrange reformulation of this minimization condition shows that a harmonic map $w$ satisfies:
\begin{equation}\label{E-L}
\Delta w  + (\log \rho )_w  w_z w_{\bar{z}} = 0
\end{equation}
which is a second order, non-linear PDE on the Riemann surface.

\end{defn}

\textit{Remark.} The existence of such a minimizer, or equivalently, a solution to \eqref{E-L}, is due to Eells-Sampson  \cite{ES}, and its uniqueness is due to Hartman and Al'ber (\cite{Hart}, \cite{Alber}).  Moreover, the harmonic map in the homotopy class of a diffeomorphism is also a diffeomorphism (\cite{Sampson}, \cite{S-Y}).

\begin{defn}[Hopf differential] A  harmonic map $w:(X, \sigma)  \to (Y,\rho)$ has the following pull-back of the metric tensor:
\begin{equation}\label{pullb}
w^\ast (\rho) = q dz^2 +  \sigma e dzd\bar{z} + \bar{q} d\bar{z}^2
\end{equation}
where $q(z) = \rho(w(z))w_{z} \bar{w_{{z}}}$. The \textit{Hopf differential} of the map is $\Phi = q(z) dz^2$.
\end{defn}

As a consequence of \eqref{E-L},  it follows that the Hopf differential of a harmonic map is a holomorphic quadratic differential. Indeed, this is a characterization of harmonic diffeomorphisms between surfaces  - see \cite{Sampson} for details.\\

In the next subsection, we shall describe estimates for the image of the harmonic map in terms of the Hopf differential. We first record a useful estimate for the energy of the harmonic map (see equation (3.4) of \cite{Minsky}):

\begin{lem}[Energy bound]\label{ebound} Let  $w:(X,\sigma) \to (Y,\rho)$ be a harmonic map with Hopf differential $\Phi$. Then for any subsurface $\Sigma \subset X$, the energy of the restriction of $w$ satisfies:
\begin{equation}\label{enorm}
2\lVert \Phi \rVert_\Sigma \leq \mathcal{E}(w\vert_{\Sigma}) \leq 2\lVert \Phi \rVert_\Sigma + 2\pi \chi(\Sigma)
\end{equation}
where $\lVert q \rVert_\Sigma = \int_\Sigma \lvert q \rvert $ is the mass of the quadratic differential when restricted to $\Sigma$.

\end{lem}

\subsection{Geometric estimates}

In this subsection we describe how estimates for the behaviour of a harmonic map $w:(X, \sigma) \to (Y,\rho)$ and the geometry of its image can derived from the Hopf differential $\Phi$. These estimates  arise from certain  Bochner-type equations derived from \eqref{E-L} that we shall not state here. We refer to section 3 of \cite{Minsky}, \cite{Wolf2} and the useful summary in \cite{Han-Remarks} for details.  \\

In what follows, we choose coordinates $\xi = x+iy$ in which the holomorphic quadratic differential $\Phi = qdz^2$ is $d\xi^2$. These \textit{canonical coordinates} for the quadratic differential are obtained by the coordinate change $z \mapsto \xi = \int^z \sqrt q$.

Moreover, we equip the domain surface with the conformal metric $\sigma = \lvert \Phi \rvert$ which is the singular-flat metric induced by the quadratic differential (that we shall sometimes refer to as the $\Phi$-metric).

\begin{defn}[Induced foliations] The $x$-direction in the canonical coordinates is the \textit{horizontal direction} induced by $\Phi$, and the $y$-direction is the \textit{vertical direction}. These determine the horizontal (resp. vertical) foliation denoted by  $\Phi^h$ (resp. $\Phi^v$) on the surface, which has singularities at the set of zeroes  $\Lambda$ of $\Phi$, and a \textit{measure} on transverse arcs determined by vertical (resp. horizontal) distances between their endpoints. 
\end{defn}

In the canonical coordinates for the Hopf differential,  the metric pullback \eqref{pullb} becomes:
\begin{equation}
w^\ast (\rho) = (e+2) dx^2 + (e-2)dy^2
\end{equation}
which implies that the horizontal (resp. vertical) directions are the directions of maximum (resp. minimum) stretch for the map $w$. 

For the following finer estimate, see \S3.3 of \cite{Minsky}:

\begin{prop}\label{gest}
Let $w:X\to (Y,\rho)$ be a harmonic diffeomorphism with Hopf differential $\Phi$. 
There exists constants $C, \alpha>0$ such that 

\begin{itemize}
\item any  horizontal segment $\gamma_H$ on $X$ of length $L$ in the $\Phi$-metric  is mapped to an arc on $Y$ of geodesic curvature less than $Ce^{-\alpha R}$ and of length $ 2L \pm Ce^{-\alpha  R}$, and

\item  any vertical  segment $\gamma_V$ on $X$ of length $L$ in the $\Phi$-metric  is mapped to an arc of length  less than $L \cdot Ce^{-\alpha  R}$,
\end{itemize}
where $R>0$ is the $\Phi$-radius of an embedded disk containing either segment, that does not contain any singularity.
\end{prop}

\textit{Example.}  For a harmonic map $h:\mathbb{C} \to (\mathbb{D}, \rho)$, where the target is the Poincar\'{e} disk equipped with the hyperbolic metric, with Hopf differential $\Phi = z^n dz^2$, the image $h(\mathbb{C})$ is an ideal polygon with $(n+2)$ ideal vertices. Here the horizontal and vertical foliations on $\mathbb{C}$ induced by $\Phi$ are the pullback of the horizontal and vertical lines by the branched covering $z \mapsto z^{(n+2)/2}$. The domain is thus divided into $(n+2)$ sectors that are Euclidean half-planes foliated by horizontal leaves; the image of  such leaves far from the singularity is close to a geodesic side in the image ideal polygon. See Figure 1 - for details, see section 3 of \cite{HTTW}.

\begin{figure}
\begin{center}

\includegraphics[scale=.5]{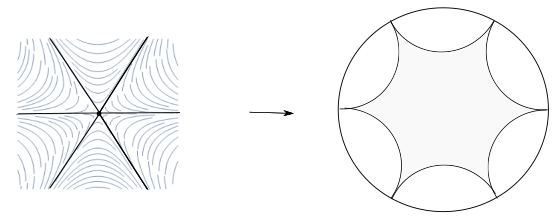}
\caption{A harmonic map from $\mathbb{C}$ to the Poincar\'{e} disk, with Hopf differental $\Phi = z^4dz^2$. The canonical coordinates for $\Phi$ are obtained by the branched covering $z \mapsto \xi = \frac{1}{3}z^3$ which results in the branched structure for the horizontal foliation (shown on the left). The image of the map is an ideal hexagon (shown on the right). }

\end{center}
\end{figure}

\subsection{Laminations and train-tracks}
In this subsection and the next we discuss  briefly some of the objects and notions pertaining to hyperbolic surfaces, that we shall need. For a more extensive treatment of geodesic laminations, we refer to \cite{CassBl} and \cite{Bon-Stony}, and for train-tracks, we refer to \cite{PenHar}.

\begin{defn}\label{lam}  A \textit{geodesic lamination}  $\lambda$ on a hyperbolic surface is a closed set that is a union of simple disjoint geodesics.  A \textit{measured lamination} is a geodesic lamination together with a measure $\mu$ supported on $\lambda$ that assigns a non-negative measure to arcs transverse to $\lambda$. 
\end{defn}

\textit{Remark.} The notions of a measured foliation (like the induced foliation of a quadratic differential) and measured lamination as above are equivalent: in one direction, ``straightening" the leaves of the measured foliation, that is, considering their geodesic representatives, yields the leaves of the lamination.

\begin{defn}\label{ttrack}  A \textit{train-track} on a surface $S$ of finite type is an embedded graph $T$ with $C^1$-smooth edges (called \textit{branches})  such that 

\begin{itemize}

\item[(a)]  there is a partition of the set of half-edges at every vertex (or \textit{switch}) into \textit{incoming} and \textit{outgoing} edges, such that half-edges of the same type enclose a ``corner" of zero angle, and differing types differ by an angle $\pi$,

\item[(b)]  the complementary regions of $T$ are not disks, or disks with  $1$ corner (monogons) or disks with $2$ corner (bigons), or once-punctured monogons or annuli. 

\end{itemize}

A \textit{weighted} train-track is an additional assignment of non-negative real numbers to each branch such that at each vertex, the sum of the weights of the incoming edges equals the sum of the weights of the outgoing edges.
\end{defn}

\textit{Remark.} A weighted train-track introduced above is a combinatorial object that encodes the information of a measured geodesic lamination when the surface $S$ is equipped with a hyperbolic metric.  To obtain the lamination, one first defines a measured foliation by  replacing each branch with a foliated band of width equal to the weight;  at the switches the additivity of weights ensures that the leaves of adjacent branches can be joined up. A  geodesic lamination with respect to the hyperbolic metric  is then obtained by ``tightening" the leaves of this foliation to their geodesic representatives. See Figure 2, and \cite{Kap} for details. \\

\begin{figure}
\begin{center}

\includegraphics[scale=.6]{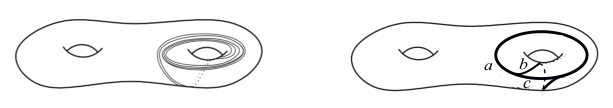}
\caption{A measured geodesic lamination (left) and a weighted train-track (right). The weights satisfy $a= b+c$.}

\end{center}
\end{figure}

Some more notions we need are as follows:

\begin{defn} An \textit{$\ep$-straight train-track} is a realization of the train-track graph in a hyperbolic surface such that the branches are geodesic, and of length at least $1$, and the angle between successive branches is less than $\ep$.
\end{defn}

\begin{defn}[Bi-recurrence] A train-track $T$  on $S$ is called bi-recurrent if 
\begin{itemize}
\item[(a)] there exists a choice of positive weights on each branch satisfying the additivity conditions, 

\item[(b)]  there is a multicurve  $\gamma$ that intersects $T$ efficiently, \textit{i.e.} $S \setminus (T\cup \gamma)$ has no bigon, and such that $\gamma$ intersects each branch at least once. 

\end{itemize}
\end{defn}

\begin{defn} A geodesic lamination $\lambda$ is said to be \textit{carried} by a bi-recurrent train-track $T$ if it is the lamination corresponding to a choice of positive weights of the branches.
\end{defn}

\textit{Remark.} Laminations carried by the same train-track defines an open set in the space of measured laminations $\mathcal{ML}$ (see Chapter 3 of \cite{PenHar}).\\

Finally, the following result can be culled from the general theory:

\begin{lem}\label{ttlem} For a bi-recurrent train-track $T$ there exists a choice of positive weights such that the corresponding lamination $\lambda$  is minimal, that is, each leaf is dense in $\lambda$.
\end{lem}

\begin{proof}[Proof sketch] 

By Theorem 1.7.12 of \cite{PenHar}, each such choice of positive weights on the branches corresponds to a unique lamination on the smallest subsurface $S_T$ supporting the train-track $T$. The image of this map is an open set in the space  $\mathcal{ML}$ of measured laminations on $S_T$ (see Chapter 3 of \cite{PenHar}), and it is well-known that a generic lamination in that space is minimal (\textit{c.f.} \cite{Keane}). 
\end{proof}

\subsection{Crowned hyperbolic surfaces}

\begin{defn}[Crown ends] A \textit{crown end} for an (open) hyperbolic surface is an annular region bordered on one side by a collection of bi-infinite geodesics that are arranged in a cyclic manner, such that each adjacent pair bounds a ``boundary cusp". (See pg. 63 of \cite{CassBl}.)  
\end{defn}

\begin{defn}
A \textit{crowned hyperbolic surface} is hyperbolic surface of finite type, having at least one crown end, and possibly some closed geodesic boundary components. See Figure 3. 
\end{defn}

\textit{Remarks.} (i) As mentioned in the Introduction, an ideal hyperbolic polygon with finitely many sides can be considered to be a  hyperbolic surface of genus zero and one crown end. \\
(ii) The metric completion $\widehat{Y \setminus \lambda}$ of the complement of a geodesic lamination $\lambda$ on $Y$ is a crowned hyperbolic surface in the sense above-- see, for example, page 7 of \cite{Bon-Stony}. \\

\begin{figure}
\begin{center}

\includegraphics[scale=.5]{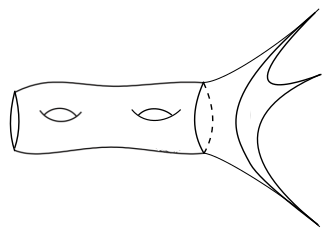}
\caption{A hyperbolic surface with a geodesic boundary and a crown end.}

\end{center}
\end{figure}

Analogous to the notions in Definition \ref{pend}, we can define: 

\begin{defn}
A \textit{truncation} of a crown end is obtained by truncating each boundary cusp along a horocyclic arc. 
The \textit{metric residue} of a crown end with an even number of boundary cusps is the alternating sum of the remaining geodesic segments of the geodesic lines around the end.  As before, this is a real number, which is defined up to sign, and independent of the truncation. For a crown end with an odd number of boundary cusps, the metric residue is defined to be zero. 
\end{defn}

It is easy to verify the following (\textit{c.f.} \S2.3 of \cite{GupHar}):

\begin{prop}\label{wild}  A crowned hyperbolic surface $Z$  of genus $g\geq 1$ and $k$ crown ends, each with  $m_1,m_2,\ldots m_k$ boundary cusps and $l$ additional geodesic boundary components,  is determined by $6g-6 +3l + \sum\limits_{i=1}^k (m_i+3)$ real parameters. 
 
 In the case of genus $g=0$, the space of ideal polygons with $m\geq 3$ cusps is determined by $m-3$ real parameters.
\end{prop}

\begin{proof}[Proof sketch]
A possible set of parameters would be \textit{shear} parameters determined by an ideal triangulation of $Z$. (See \cite{Bon-Shear} for an account of the case of a \textit{closed} surface.) 
Here, the geodesic lines determining the ideal triangles of the triangulation will include the geodesic sides of each crown end, and each geodesic boundary component necessarily has a  geodesic line spiralling into it. Note that the geodesic sides of a crown end has no shear parameter, as there is no adjacent ideal triangle. Moreover, note that the shear parameter of the geodesic line spiralling into a closed boundary geodesic is uniquely determined by the boundary length. 
The number of geodesic lines that each contributes a shear parameter is then exactly $6g-6 +3l + \sum\limits_{i=1}^k (m_i+3)$, as can be verified by an easy Euler characteristic count.

For ideal $m$-gons, the counting of shear parameters is easier: any ideal triangulation would involve a choice of  exactly $m-3$ disjoint diagonal lines. 
\end{proof}

\section{\Te rays and their conformal limits}

In this section we shall fix a marked Riemann surface $X$ in $\mathcal{T}_g$ where $g\geq 2$, and a holomorphic quadratic differential $\Psi_0$ on $X$.

\begin{defn}\label{tray} The \textit{\Te ray} determined by $(X,\Psi_0 )$ is the family of Riemann surfaces $\{X_t\}_{t\geq 0}$ such that $X_0 = X$ and $X_t$ is obtained by a quasiconformal distortion $f_t: X_0 \to X_t$ that stretches in the vertical direction, that is, $f_t(x,y) = (x,  ty)$ in the canonical coordinates $\xi = x+iy$ of the quadratic  differential $\Psi_0$. The resulting surface $X_t$ is equipped with a quadratic differential $\Psi_t$, together with its induced singular-flat metric. 
\end{defn}

\textit{Remarks.} (i)  A \Te ray as defined above is in fact a geodesic in the \Te metric on $\mathcal{T}_g$.  See, for example, \cite{Lehto} for an account.

(ii) In this article, we adopt the convention of defining a \Te ray by a \textit{vertical} stretch, such that the transverse measures of horizontal foliation $\Psi_0^h$ scale by a factor of $t$ along the ray, and the measure-equivalence class of the vertical measured foliation remains fixed. This might differ from other literature; however, the choices of ``horizontal" and ``vertical"  are easily interchangeable by replacing $\Psi_0$ by $-\Psi_0$, and the various definitions of a \Te ray are the same up to rescalings.

(iii) Given a Riemann surface and a measured geodesic lamination $\lambda$ (with respect to some hyperbolic metric) there is a unique holomorphic quadratic differential $q$ whose induced horizontal foliation is equivalent to $\lambda$ (see \cite{HubMas}).   In particular, a \Te ray  (as defined in Definition \ref{tray}) can be thought of as being defined by an initial surface and a geodesic lamination $\lambda$, instead of a holomorphic quadratic differential.\\

Recall the following notion:

\begin{defn}[Critical graph] 
A \textit{critical trajectory} of the horizontal foliation $q^h$ of a quadratic differential $q$ is a horizontal leaf that is incident on the singularity set $\Lambda$ of the differential; the union of critical trajectories defines the \textit{critical graph}  $\Gamma$. Note that $\Gamma$ is a metric graph (the finite-length edges are \textit{saddle-connections} between points of $\Lambda$) and it also acquires a ``fat-graph" or ``ribbon-graph" structure, namely, a cyclic ordering of the half-edges emanating from any vertex,  from the oriented surface in which it embeds. 
Also, note that  $\Gamma$ could have more that one connected component, and some edges of $\Gamma$ could be of infinite length; indeed, a critical leaf could be dense on the surface.  See Figure 4. 
\end{defn}

\begin{figure}
\begin{center}

\includegraphics[scale=.41]{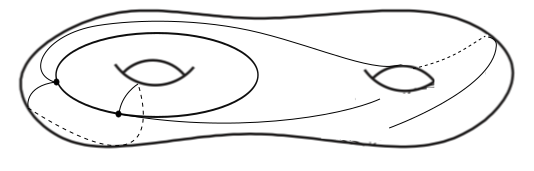}
\caption{A part of a critical graph $\Gamma$ for the horizontal foliation of a quadratic differential, which comprises horizontal trajectories emanating from the zeroes.}

\end{center}
\end{figure}

The following definition is implicit in \cite{Gup25}  (and appears in an earlier unpublished version of that article).

\begin{defn}\label{tlim} The \textit{conformal limit} $X_\infty$ of a \Te ray $X_t$ is the Gromov-Hausdorff limit of the singular-flat surfaces with respect to the basepoint set $\Lambda$ which is the set of zeroes of $\Psi_0$ (or the singularities of the induced metric).  This limiting surface $X_\infty$ might not be connected, and acquire singular-flat metrics of infinite area, as they are limits of metrics obtained by a deformation that stretches in one direction. 

Note that in the stretched metric, the critical graph $\Gamma$ of $\Psi_0$  contains $\Lambda$ and remains isometrically embedded, as the horizontal lengths do not change. However, as $t\to \infty$, any compact piece of $\Gamma$ acquires  larger and larger neighborhoods, comprising isometrically embedded Euclidean rectangles of  width (vertical distance)  tending to infinity. 

Thus, a more constructive description of the surface $X_\infty$ is the following: it is the singular-flat surface of infinite area obtained by attaching Euclidean half-planes $\mathbb{R} \times \mathbb{R}_{\geq 0}$  and half-infinite Euclidean cylinders  $S^1 \times \mathbb{R}_{\geq 0}$ to the critical graph $\Gamma$ of the horizontal foliation of $\Psi_0$, such that the gluing respects its ribbon-graph structure. 

\end{defn}

\subsection*{Examples.} 

 (i) The simplest example is when the quadratic differential $\Psi_0$ determining the \Te ray is a Jenkins-Strebel differential, that is, when  $\Gamma$ is a compact graph, and all  non-critical leaves of $\Psi_0^h$ are closed, and foliate cylinders. The induced singular-flat metric then comprises these Euclidean cylinders with their boundaries isometrically identified with cycles in $\Gamma$. Along the \Te ray, these cylinders stretch the vertical direction, and the conformal limit is the surface obtained by attaching half-infinite Euclidean cylinders to the components of $\Gamma$. \\

(ii) A generic quadratic differential $\Psi_0$ has a critical graph $\Gamma$ which has $4g-4$ components, each of which is an infinite ``tripod" with a single vertex, and three infinite rays. The corresponding \Te ray then has a conformal limit with $4g-4$ components, each  a singular-flat surface  obtained by attaching three half-planes to the tripod, and conformally equivalent to  the complex plane. \\

 (iii) More generally, whenever the horizontal foliation of $\Psi_0$  is minimal, then the conformal limit $X_\infty$ is obtained by attaching half-planes to the critical graph $\Gamma$ of $\Psi_0$, and thus has  a ``half-plane structure".  (See Figure 5.)\\

\begin{figure}
\begin{center}

\includegraphics[scale=.4]{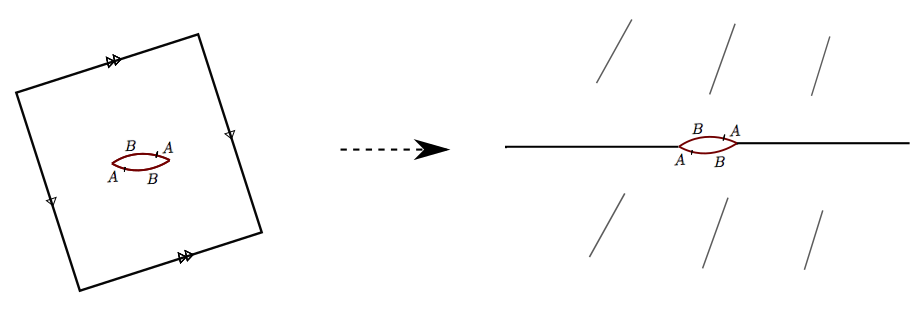}
\caption{The genus-$2$ singular-flat surface $X_0$ on the left is obtained by making a horizontal slit on a flat torus, and  identifying intervals on the resulting boundary by isometries, as shown. The sides of the square are at an irrational slope so that lines in the horizontal direction are dense on the surface. The \Te ray $X_t$ is then obtained by stretching in the vertical direction; the limit $X_\infty$ is the singular-flat surface on the right obtained by isometric identifications of intervals on the boundary of two Euclidean half-planes, as shown.}

\end{center}
\end{figure}

\subsection*{Conformal limits} 
Note that as a consequence of the construction detailed in Definition \ref{tlim}, we have:

\begin{lem}\label{tlimlem} Let $X_\infty$ be the conformal limit of a \Te ray $\{X_t\}_{t>0}$, with connected components $X_\infty^1, X_\infty^2, \ldots, X_\infty^m$.

For any $R>0$, there exists $T> 0$ such that for any $X_t$ with $t>T$, there is a decomposition 
\begin{equation}\label{eq-decomp}
X_t = X_t^1 \sqcup X_t^2 \sqcup \cdots \sqcup X_t^m
\end{equation}
with isometric embeddings $\phi_t^j: X_t^j \to X_\infty^j $ for each $1\leq j\leq m$, which exhaust each component.

\end{lem}

\medskip

More generally, we shall say:

\begin{defn}[Conformal limit]\label{clim} A Riemann surface $W_\infty$ is a \textit{conformal limit} of a family of Riemann surfaces $\{W_t\}_{t\geq 0}$ if 
there is a family of $K_t$-quasiconformal embeddings $\phi_t:W_t\to W_\infty$ such that $K_t\to 1$ as $t\to \infty$, and their images exhaust $W_\infty$, that is,  $\bigcup\limits_t \phi_t(W_t)  = W_\infty$. 

(The above definition also holds for a \textit{sequence} of surfaces, in which case the parameter $t$ takes only positive integer values.) 
\end{defn}

\textit{Remark.} As a consequence of Lemma \ref{tlimlem} the components of the singular-flat surface $X_\infty$  constructed in Definition \ref{tlim} are indeed the conformal limits, in the sense above, of the pieces of the decomposition \eqref{eq-decomp} of $X_t$,  along the \Te ray. This justifies calling $X_\infty$ the conformal limit of the \Te ray $X_t$.  \\

The following lemma is immediate from the definition:

\begin{lem}\label{limlem} 
Suppose $W_\infty$ is a conformal limit of a family (or sequence) of Riemann surfaces $\{W_t\}_{t\geq 0}$ and $V_\infty$ is a conformal limit of a family (or sequence)  of Riemann surfaces $V_t \subset W_t$, then there is a conformal embedding $f:V_\infty \to W_\infty$.
\end{lem}
\begin{proof}
In what follows let $i_t: V_t \to W_t$ be the inclusion map, which is a conformal embedding. 
Fix a compact set $K \subset V_\infty$.
For any $\ep>0$, and sufficiently large $t$, there is a $(1+\ep)$-quasiconformal embedding $\psi_t:V_t\to V_\infty$ whose image includes $K$.  We can pick $t$ large enough such that there is also a  $(1+\ep)$-quasiconformal embedding $\phi_t:W_t\to W_\infty$. 
Then the composition $\phi_t \circ i_t \circ \psi_t^{-1}: K \to W_\infty$ is a $(1 + O(\ep))$-quasiconformal embedding.
Letting $\ep\to0$, we conclude that there is a conformal embedding of $K$ into $W_\infty$, call it $f$.
Note that when this construction is applied to a further compact subset $K_1 \subset K$, we obtain a conformal embedding that is a restriction of $f$. 
Since $K$ was arbitrary, we have in fact defined a conformal embedding $f:V_\infty \to W_\infty$, as desired.
\end{proof}

We also note: 

\begin{defn}[Limit of maps]\label{conv} Let $W_t$ be a family (or sequence) of Riemann surfaces with conformal limit $W_\infty$. A family  (or sequence) of smooth maps $f_t:W_t \to Y$ is said to \textit{converge} to $f:W_\infty \to Y$ if for any compact set $K \subset W_\infty$, the maps $f_t \circ (\phi_t\vert_K)^{-1}:K \to Y$ converge uniformly as $t\to \infty$.  (Note that $K$ lies in the image of $\phi_t$, and the composition makes sense, for all sufficiently large $t$.)
\end{defn}

\medskip

\subsection*{Planar ends} 
We end this section with the following definitions  (culled from \cite{Gup25}) that will be useful in \S5.

 Let $X_\infty$ be the infinite-area singular flat surface obtained as a conformal limit of a \Te ray.

\begin{defn}\label{pend}[Planar and cylindrical ends]
A \textit{planar end} of $X_\infty$ is a neighborhood of a puncture, comprising a collection of Euclidean half-planes attached to each other in a cyclic order by isometries on boundary rays,  which is conformally a punctured unit disk. 
A \textit{cylindrical end} of $X_\infty$ is a neighborhood of a puncture which is a half-infinite Euclidean cylinder. 
\end{defn}

\begin{defn}
A \textit{truncation} of a planar end is obtained by removing a polygonal region  (that is, bounded by alternating horizontal and vertical segments) from each half-plane such that the remaining flat surface has a polygonal boundary. 
(Note that this is slightly more general than the notion introduced in \cite{Gup25}, where the region removed is a rectangle on each half-plane of the planar end.)  Similarly, a truncation of a cylindrical end is obtained by removing an end of the cylinder bounded by a geodesic circle. 
\end{defn}

\begin{defn}
The \textit{metric residue} of a planar end $\mathcal{E}$ with an even number of half-planes is  obtained by considering a truncation of $\mathcal{E}$, and taking the alternating sum of the total horizontal  length of the sides in each half-plane. (It can be checked that this real number, which is defined up to sign, is independent of the truncation.) For a planar end with an odd number of half-planes, the metric residue is defined to be zero. 
Finally, the metric residue for a cylindrical end is defined to be the length of the boundary circle.  
(These are called  ``end data" in \cite{Gup25}.) 

\end{defn}

\section{Proof of Theorem \ref{thm1}: Harmonic limits}
Recall that we have a sequence of harmonic diffeomorphisms $h_t:X_t\to Y$, where $Y$ is a fixed hyperbolic surface, and $X_t$ is a \Te ray  determined by an initial Riemann surface $X_0$ and a  (measured) geodesic lamination $\lambda$ on Y (\textit{c.f.} Definition \ref{tray} and the third remark following it).\\

Let $\Phi_t$ be the Hopf differential for  $h_t$, and let $\Phi_t^h$ be its horizontal foliation. In what follows $\eta_t$ will be the measured geodesic lamination on $Y$ that is measure-equivalent to $\Phi_t^h$. \\

Here is a sketch of the proof:\\

\begin{itemize}

\item[\textit{Step 1.}]  Fix an $\ep>0$.  We apply Minsky's ``polygon decomposition" of the quadratic differential metric $\Phi_t$ on $X_t$, for all sufficiently large $t>0$.  The polygon pieces are $\mathcal{P}_{R,t}$ for some large $R$ (depending only on $\ep$). We show, using Minsky's work,  that the image under $h_t$ of the complement $\mathcal{P}_{R,t}^c$
is an $\ep$-neighborhood of the lamination $\lambda$ on $Y$.\\

\item[\textit{Step 2.}]   We show that any such family of subsurfaces  $\mathcal{P}_{R,t}$ converge to Riemann surface with boundary as $t\to \infty$ (note that $R$ is fixed here).  To do this, we show that the extremal lengths of curves supported in $\mathcal{P}_{R,t}$  remain uniformly bounded above and below.  Note that the energy of $h_t$ restricted to $\mathcal{P}_{R,t}$ remains uniformly bounded as $t\to \infty$.\\

\item[\textit{Step 3.}]  Now, for a sequence $\ep_n \to 0$, we pick $R_n\to \infty$ for which the previous steps hold. In particular, \textit{Step 2} shows that  $\mathcal{P}_{R_n,t}$ has a limiting surface (as $t\to \infty$) that conformally embeds in the conformal limit $X_\infty$.  We show that the conformal limit for a suitable sequence $\mathcal{P}_{R_n,t_n}$ is $X_\infty$.\\

\item[\textit{Step 4.}]  Standard convergence results then show that the sequence of harmonic maps $h_{t_n}$ sub-converges to a limiting harmonic map on each compact set. The domain is the conformal limit $X_\infty$, and the target is the entire complement  $Y\setminus \lambda$, whose metric completion is a crowned hyperbolic surface. 

\end{itemize}

\subsection{Step 1: Polygonal decomposition of Minsky}

Suppose $\eta_t$ is the measured lamination corresponding to $\Phi_t^h$ on $Y$ ( \textit{c.f.} the remark following Definition \ref{lam}).  The compactness of the space $\mathcal{PML}$ of \textit{projectivized} measured laminations (see Chapter 3 of \cite{PenHar}), then implies that  after passing to a subsequence, the projective classes $[\eta_t]$ converge. Now scaling the measure does not affect the underlying geodesic lamination, and hence the underlying geodesic laminations for this subsequence converges to a geodesic lamination $\eta$ on $Y$.

First, we apply Minsky's arguments in \S4, 5 of his paper to achieve the following decomposition of $X_t$ equipped with the quadratic differential $\Phi_t$:

\begin{prop}[Theorem 5.1  of \cite{Minsky}]\label{decomp}  For any $R>0$, for all sufficiently large $t$, the surface $X_t$ can be decomposed into boundary-convex \textit{polygonal subsurfaces} $\mathcal{P}_{R,t}$  that contains the $R$-neighborhood of the set of zeroes of the quadratic differential $\Phi_t$. These polygonal subsurfaces have area $O(R^2)$ and  boundary length $O(R)$ where the constants depend only on the genus of the surface. 
\end{prop}

\textit{Remark.} The term ``polygonal" comes from the fact that each boundary component comprises alternate horizontal and vertical segments. Some boundary components are \textit{straight}, that is, are circles bounding a flat cylinder in the complement $\mathcal{P}_{R,t}^c$. 

(See Theorem 5.1 of \cite{Minsky} for more properties of these polygonal subsurfaces, some of which we might be using in this paper without explicitly stating.) 
 \\

The following is then a consequence (see \S6 of Minsky's paper).

\begin{lem}[Lemma 6.6 of \cite{Minsky}]\label{tt} For $\ep>0$, there is a choice of $R_0(\epsilon)>0$ such that for any  $R>R_0$ the image under $h_t$ of the complement of the polygonal subsurfaces $\mathcal{P}_{R,t}$ on $X_t$  is $\ep$-close to an $\ep$-straight train-track $\tau_t$. Moreover,  $\tau_t$ can be enlarged by adding branches,  to yield a train-track carrying $\eta_t$.
\end{lem}

Finally, Lemma 8.2 of \cite{Minsky} asserts:

\begin{lem}\label{liml} The limit $\eta$ of the geodesic laminations underlying $\eta_t$ on $Y$ equals $\lambda$, the geodesic lamination underlying the measured lamination that determines the \Te ray. 
\end{lem}

\medskip

As a consequence, we observe:

\begin{lem}\label{cor-im} In Lemma \ref{tt}, the complement of the polygonal subsurfaces has an image  $h_t(\mathcal{P}_{R,t}^c)$ that is an $\ep$-neighborhood of $\lambda$ on $Y$.
\end{lem}

\begin{proof}

In Minsky's proof of Lemma \ref{tt},  there are three kinds of additional branches that are used to enlarge the $\ep$-straight train track $\tau_t$:\\
  (i) branches that are the image of leaves of $\Phi_t^h$ contained in $\mathcal{P}_{R,t}$, that are arcs between vertical segments of the polygonal boundary (see Lemma 6.3 of \cite{Minsky}), \\
  (ii) branches that are images of arcs spiralling along the flat cylinders in the complement  $\mathcal{P}_{R,t}^c$, each isotopic to a leaf of $\Phi_t^h$ (see Lemma 6.4 of \cite{Minsky}), and\\
  (iii) branches that constitute train-track components that are disjoint from $\tau_t$, and lie in the image of $\mathcal{P}_{R,t}$ (see Lemma 6.6 of \cite{Minsky}).
  
  For each in turn, we observe:
 
 (1) Let  $\alpha$ be such an arc contained in the polygonal subsurface $\mathcal{P}_{R,t}$ and having endpoints on the boundary $\partial \mathcal{P}_{R,t}$, that appears as a branch of the train-track carrying $\Phi_t^h$, for all sufficiently large $t$.  The transverse measure of such a branch $b$ is the measure (or the vertical distance in the $\Phi_t$-metric) across a band of leaves of $\Phi_t^h$ homotopic to $\alpha$, contained in  $\mathcal{P}_{R,t}$. This transverse measure is  then uniformly bounded, as $t\to \infty$, since the diameter of  $\mathcal{P}_{R,t}$ remain uniformly bounded (see Proposition \ref{decomp}).
 In particular, the \textit{projectivized} measured foliations $[\eta_t]$ will have the weight of the branch $b$ tending to zero, which implies that the train-track for the limiting lamination $\eta$ does not, in fact, need the branch $b$. Since $\eta$ equals $\lambda$ (Lemma \ref{liml}) the same is true for the latter lamination.

 (2) The construction in \cite{Minsky} of the ``polygonal decomposition" of Proposition \ref{decomp} implies that for sufficiently large $R$, any ``straight" boundary component of $\mathcal{P}_{R,t}$ (bounding a flat cylinder in the complement) has a flat cylindrical collar of large modulus.  Since the extremal length of a core curve of such a  flat cylindes tends to $0$ as $t\to \infty$, the image of such a core curve must be homotopic to a closed geodesic component $\gamma$  of the lamination $\lambda$.  The image of the leaves of $\Phi_t^h$ in the flat cylinder then spiral closely around $\gamma$ (see the final part of \cite{Minsky}) and in particular, remain in its $\ep$-neighborhood.

 (3) Similar to (1), a branch $b$ for $\Phi_t^h$ that lies completely inside $\mathcal{P}_{R,t}$ for all sufficiently large $t$, must have transverse measure uniformly bounded. As a consequence, the limit $\lambda$ of the geodesic laminations underlying the projectivized laminations $[\eta_t]$ is carried by a train-track that does not require $b$.  This holds for each branch of the train-track component, and hence the limiting lamination $\eta$ (which equals $\lambda$) does not have such a component.

This concludes the proof.
\end{proof}

\begin{figure}
\begin{center}

\includegraphics[scale=.4]{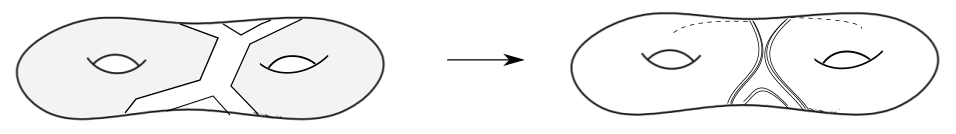}
\caption{The harmonic map $h_t$ takes the complement of the polygonal subsurfaces (shown shaded) on $X_t$ to a neighborhood of the lamination $\lambda$ on $Y$.}

\end{center}
\end{figure}

\subsection{Step 2: Convergence for fixed $R$}

We continue with the choice of $\ep>0$ as in the previous section, and we fix a (sufficiently large)  $R > R_0(\ep)\gg 0$ (see Lemma \ref{tt}). Then from the previous section, we have a family of (possibly disconnected) singular-flat surfaces $\mathcal{P}_{R,t} \subset X_t $ for all $t$ sufficiently large.

We know from Lemma \ref{cor-im} that for $t$ sufficiently large, the image of the complement $X_t \setminus \mathcal{P}_{R,t}$ is the subsurface that supports the image lamination $\lambda$.

Let $Y_0$ denote that subsurface of $Y$, and let $X_0$ be the subsurface of $X$ homeomorphic to $Y_0$ (recall that all the harmonic maps $h_t$ are diffeomorphisms). 

Then the singular-flat surface $\mathcal{P}_{R,t}$ is homeomorphic to the (possibly disconnected) surface-with-boundary $X\setminus X_0$.  

The purpose of this section is to show that it in fact, when one considers the underlying conformal structures, it converges  (as $t\to \infty$) to a (possibly disconnected)  \textit{Riemann surface with boundary} that is topologically equivalent to $X\setminus X_0$.

We start with:

\begin{lem}\label{ub} For all (sufficiently) large $t\gg 0$, any essential loop $\gamma$ in $X\setminus X_0$ 
has a uniformly bounded extremal length in $\mathcal{P}_{R,t}$.  
\end{lem}

\begin{proof}

First, consider the case of a non-peripheral loop $\gamma$.

Note that since $\gamma$ is supported in $X\setminus X_0$, it does not intersect $\lambda$, and hence its extremal length on $X_t$ remains uniformly bounded below as $t\to \infty$.  In other words, the only lamination that ``pinches" along the \Te ray is $\lambda$ -- the argument in the proof of Lemma 8.3 of \cite{Minsky}, together with the ``analytic definition" of extremal length  (see \eqref{ext}), provides such a uniform lower bound.

Since $\mathcal{P}_{R,t} \subset X_t$, the extremal lengths satisfy 
 $$\text{Ext}_{\mathcal{P}_{R,t}}(\gamma) \geq  \text{Ext}_{X_t}(\gamma)$$ 
 providing a uniform lower bound for the extremal lengths of $\gamma$ on $\mathcal{P}_{R,t}$.

Note that there is a uniform upper bound for the extremal length of $\gamma$ on $X_t$ as $t\to 0$ also  (the initial surface admits an annulus with core curve $\gamma$ that persists as the metric stretches - see Lemma 8.3 of \cite{Minsky}). 

The uniform upper bound for the extremal length on $\mathcal{P}_{R,t}$ is then a consequence of:\\

 \textit{Claim.The ratio of the extremal lengths  of $\gamma$ on $\mathcal{P}_{R,t}$ and $X_t$ is uniformly bounded.}\\
 \textit{Proof of claim.} For all sufficiently large $t$, the boundary of $\mathcal{P}_{R,t}$ has a collar of large modulus contained inside the subsurface, comprising wide Euclidean rectangles corresponding to the branches for the lamination $\eta_t$ that are adjacent to the boundary. (The width of the rectangle equals the weight of the branch, or the transverse measure of $\eta_t$ across the branch, that tends to infinity as $t\to \infty$.) In particular, the modulus of this peripheral collar is greater than $m_0$, where $m_0$ is as in Lemma 8.4 of \cite{Minsky} (see also Lemma 4.2 of \cite{MinskyProduct}). The claim then follows from an application of Minsky's lemma. 
$\qed$\\

If $\gamma$ is peripheral, the argument is easier:  the lengths of the horizontal sides of the polygonal boundary component  (subset of  $\partial \mathcal{P}_{R,t}$) homotopic to $\gamma$  are bounded above by  $c_1R$, and the  lengths of the vertical sides are bounded below  by $c_2R$, for some constants $c_1,c_2>0$.  
Hence  one can construct a boundary collar comprising  Euclidean rectangles  of definite modulus,  sharing sides with the sides of the polygonal boundary. In particular,  there is an embedded peripheral collar of $\gamma$ of definite modulus, embedded in $\mathcal{P}_{R,t}$.

This gives a uniform upper bound for the extremal length of $\gamma$ on  $\mathcal{P}_{R,t}$  (as $t\to \infty$), by the ``geometric definition" of extremal length, namely,
\begin{equation*}
\text{Ext}_{W}(\gamma)  = \inf_A \frac {1}{\text{Mod}(A)}
\end{equation*}
where $A \subset W$ is an embedded annulus with core curve homotopic to $\gamma$.

The uniform lower bound on the extremal length of $\gamma$ on  $\mathcal{P}_{R,t}$ arises from the ``analytic definition" of extremal length. That is, recall that 
\begin{equation}\label{ext}
\text{Ext}_{W}(\gamma)  = \sup_\rho \frac {l_\rho^2}{A_\rho}
\end{equation}
where $l_\rho$ is the length of $\gamma$ and $A_\rho$ the area of $W$ with respect to a conformal metric $\rho$. 
Applying this to $\text{Ext}_{\mathcal{P}_{R,t}}(\gamma)$ provides a  lower bound determined by the singular-flat  conformal metric that $\mathcal{P}_{R,t}$ is already equipped with. The quantities $l_\rho$ and $A_\rho$ are uniformly bounded (above and below) as $t\to \infty$ because of the uniform control on the diameters of $\mathcal{P}_{R,t}$ and lengths of boundaries.
\end{proof}

This yields the desired sub-convergence:

\begin{lem}\label{plim}
After passing to a subsequence, the surfaces $\mathcal{P}_{R,t}$ converge to a singular-flat surface $\mathcal{P}_{R,\infty}$. 
\end{lem}

\begin{proof}
Since the fixed quantity $R>0$ determines the diameter of each singular-flat surface $\mathcal{P}_{R,t}$, they have a uniform bound in area, and necessarily subconverge to a singular-flat surface. Thus, we only need to confirm that the limiting surface is topologically the same as its approximates, that is, no loop  ``pinches" in the limit. This is guaranteed by Lemma \ref{ub}.
\end{proof}

\smallskip

\subsection{Step 3: Conformal limit is $X_\infty$}

In the previous section, we had fixed an $\ep>0$, and a corresponding $R>R_0(\ep)$ (see Lemma \ref{tt}) , and considered the singular-flat surfaces $\mathcal{P}_{R,t} \subset X_t$ as $t\to \infty$.

In this section, we first pick a sequence $\ep_n \to 0$, and a corresponding sequence $R_n \to \infty$, satisfying (a) $R_n > R_0(\ep_n)$, and  (b) $R_{n+1} > R_n$.

Lemma \ref{plim} then implies that for each $n$, there is a sequence $t_m \to \infty$ such that the singular-flat surfaces $\mathcal{P}_{R_n,t_m}$ converge to a singular-flat surface $\mathcal{P}_{R_n,\infty}$. 

Note that since the surface  $\mathcal{P}_{R_n,t_m} \subset X_{t_m}$, then by Lemma \ref{limlem} there is a conformal embedding  $\mathcal{P}_{R_n,\infty}$ in $X_\infty$.

In particular, we can choose an increasing sequence $t_n \to \infty$ such that  for each $n$, there is a $(1 + \ep_n)$- quasiconformal embedding $\phi_n:  \mathcal{P}_{R_n,t_n} \to X_\infty$.\\

The main observation is:

\begin{lem}\label{exh} The images of the sequence of embeddings $\phi_n$, defined above, forms an exhaustion of $X_\infty$. That is, for any compact subsurface $K \subset X_\infty$, there is a sufficiently large $N>0$ such that for each $n>N$, the image of $\phi_n:  P_{R_n,t_n} \to X_\infty$ contains $K$.
\end{lem}
\begin{proof}

Choose a compact exhaustion $\cdots K_m \subset K_{m+1} \subset \cdots$  of $X_\infty$, such that each $K_m$ is homotopic to the subsurface $X\setminus X_0$.  (In particular,  $K_{m+1}$ is obtained from $K_m$ by attaching an annulus to the boundary components.) 
It suffices to check that the image of $\phi_n$ contains any fixed $K_m$, for all sufficiently large $n$. 

Note that by Lemma \ref{cor-im}, each $\mathcal{P}_{R_n,t_n}$ is homotopic to the subsurface $X\setminus X_0$,  and so is the image of $\phi_n$. 

Since $R_n \to \infty$, each boundary component of the subsurface $\mathcal{P}_{R_n,t_n}$ has a collar of modulus  $M_n$, where $M_n \to \infty$ as $n\to \infty$. The same is true for the sequence of images under the almost-conformal embeddings $\phi_n$.  In particular, the union of these images must be conformally a union of punctured disks (one for each boundary component), and hence, must exhaust the ends of $X_\infty$. 
 \end{proof}

From Definition \ref{clim}, we immediately obtain:

\begin{cor} The conformal limit of the sequence of singular-flat surfaces $\mathcal{P}_{R_n,t_n}$ ($n\geq 1$) is $X_\infty$.\end{cor}

\subsection{Step 4: Convergence of harmonic maps}

Finally, we can show:

\begin{prop} Let $t_n$ ($\to \infty$) be the sequence chosen in \S4.3. The maps $h_{t_n}:X_{t_n} \to Y$ converge to a harmonic map $h:X_\infty \to Y$ whose image is $Y \setminus \lambda$, which is the interior of $\widehat{Y \setminus \lambda}$.
\end{prop}

(Here, the convergence of maps is in the sense of Definition \ref{clim}.) 

\begin{proof}

Fix an $R>0$. 
As a consequence of  Lemma \ref{ebound},  the restrictions of $h_{t_n}$ to $\mathcal{P}_{R,t_n}$ have uniformly bounded energy, as $t_n\to \infty$.

By Lemma \ref{plim}, these subsurfaces $\mathcal{P}_{R,t_m}$ sub-converge to a singular-fat surface $\mathcal{P}_{R, \infty}$.

Note that for all sufficiently large $n$, we have  $R_n>R$.

The $(1+ \ep_n)$-quasiconformal embeddings $\phi_n: \mathcal{P}_{R_n,t_n} \to X_\infty$, when restricted to $\mathcal{P}_{R,t_n}$, then subconverge to a \textit{conformal} embedding  $\phi: \mathcal{P}_{R, \infty} \to X_\infty$, since $\ep_n \to 0$ as $n\to \infty$.

The compositions $h_{t_n} \circ \phi_n^{-1}$ restricted to the image of $\mathcal{P}_{R,\infty}$ then form a equicontinuous family (the equicontinuity of the harmonic maps is a consequence of the uniform energy bound noted above, and the fact that $\phi_n$ converge uniformly). 
Moreover, this family is uniformly bounded as the target $Y$ is fixed, and consequently has fixed diameter. 
Hence by Arzela-Ascoli, there is a limiting map, defined on the subset $\phi(\mathcal{P}_{R,\infty})$ of $X_\infty$, which  is also harmonic, since harmonicity is preserved under conformal reparametrizations of the domain surface.

Finally, as a consequence of Lemma \ref{exh}, for the increasing sequence $R_n \to \infty$, the images $\phi(\mathcal{P}_{R_n,\infty})$ exhaust $X_\infty$ as $n\to \infty$, and by the above argument there is a limiting harmonic map $h:X_\infty \to Y$. 

By construction, the complement of the image of $\mathcal{P}_{R_n, t_n}$ under $h_{t_n}$ is an $\ep_n$-neighborhood of $\lambda$ (see Lemma \ref{cor-im}). Since $\ep_n\to 0$, the images of $h_{t_n}$ exhaust the  complement of $\lambda$.
Hence the image of the limiting map $h$ is exactly the complement of the lamination $\lambda$ on $Y$, that is, it is the interior of the crowned hyperbolic surface which is the metric completion $\widehat{Y \setminus \lambda}$. 
\end{proof}

Observe that:

\begin{lem} The limiting harmonic map  $h:X_\infty \to Y$ above has poles of order $2$ at the punctures corresponding to the cylindrical ends, and poles of order $m\geq 3$ at a puncture corresponding to a planar end having $m$ half-planes. 
\end{lem}
\begin{proof}
By Lemma \ref{cor-im}, for all sufficiently large $t$ the singular-flat surfaces $\mathcal{P}_{R,t}$ (arising from the $\Phi_t$-metric) have boundary components that are straight (for boundary components whose complement bounds a flat cylinder) or polygonal (the horizontal sides of which are determined by the branches of the train-track for $\lambda$ adjacent to the image of the boundary).  The same holds for the limiting surfaces $\mathcal{P}_{R,\infty}$ (see Lemma \ref{plim}) which are conformally embedded in $X_\infty$.  For large $R$, the singular-flat metrics of these limits have collars of large modulus, which are,  a flat cylinder of length $O(R)$ for a straight boundary component, and a flat  rectangular annulus comprising $m$ half-annuli with  outer ``radius" $O(R)$ and inner ``radius" $O(1)$,  for a polygonal boundary component.   Here  $m$ is the number of horizontal sides in the polygonal boundary. 
 The modulus of such a boundary collar is $O(R)$ in the former case, and $O(\frac{1}{m}\ln R)$ in the latter case, and the areas are $O(R)$  and $O(mR^2)$ respectively. 
 
 Note that by Lemma \ref{ebound} and the convergence proved in the previous lemma,  the  \textit{area}  of a collar (which equals the $\Phi_t$-norm) determines bounds for the energy of the restriction of the harmonic map $h$ to the corresponding subset of $X_\infty$.   

Recall from \S4.3 that these subsurfaces $\mathcal{P}_{R,t}$ exhaust the ends of $X_\infty$ (\textit{i.e.}, neighborhoods of the punctures) as $R\to \infty$. The singular-flat metric on $X_\infty$ (arising as a limit of singular-flat surfaces along ta \Te ray) has cylindrical and planar ends, that match the ends developing in the  limit of the $\Phi_t$ metrics on $\mathcal{P}_{R,t}$, since both are determined by the structure of the train-track graph for $\lambda$. In particular, the boundary component of $\mathcal{P}_{R,t}$ with $m$ horizontal sides corresponds to a planar end with $m$ half-planes. 

 The above rough computations then determine the rate of energy growth of $h$ in the neighborhoods of the punctures of $X_\infty$. The fact that the energy tends to infinity implies that the Hopf differential of $h$ has a pole of order greater than $1$ at the punctures, and it is easy to verify that the rate of energy growth matches that of a pole of order $2$, and order $m$, for the two types of ends above.  \end{proof}

This completes the proof of Theorem \ref{thm1}.

\section{Proof of Theorem \ref{thm2}: Existence result}

Theorem \ref{thm1} provides a way to prove an existence theorem of harmonic maps from a given punctured Riemann surface $X \setminus P$  to a crowned hyperbolic surface $Z$:\\

Namely, we shall construct 

\begin{itemize}

\item[A.] A closed hyperbolic surface $Y$  of genus $g^\prime \geq g$ with a measured geodesic lamination $\lambda$  such that  $Z$ is the metric completion of one of the components $Y_0$  of $Y \setminus \lambda$, and \\

\item[B.] A  Teichm\"{u}ller ray $X_t$ in $\mathcal{T}_{g^\prime}$ that is determined by a holomorphic quadratic differential whose horizontal foliation is equivalent to $\lambda$, having a conformal limit $X_\infty$ with a component  (corresponding to the component $Y_0$) that is conformally homeomorphic to $X\setminus P$. \\

\end{itemize}

It follows from Theorem \ref{thm1} that the harmonic diffeomorphisms $h_t:X_t \to Y$ will converge to a harmonic map $h_\infty:X_\infty \to \widehat{Y \setminus \lambda}$ . The restriction of $h_\infty$  to the appropriate connected component then yields the desired harmonic map $h:X\setminus P  \to Z$, proving Theorem \ref{thm2}. \\

In the next two subsections, we shall describe the constructions for (A) and (B) above.

\subsection{Construction of $Y$ and $\lambda$}

Our proof uses the theory of train-tracks, the relevant portions of which we had briefly reviewed in \S2.3.

\subsubsection*{Single crown} 
Consider the case when $Z$ has genus $g\geq 0$, and a single crown with $m\geq 1$ boundary cusps.

We need to show:

\begin{lem} Given  a crowned hyperbolic surface $Z$ as above,  there is a hyperbolic surface $Y$ and a geodesic lamination $\lambda$  which has a complementary region, the metric completion of which is isometric to $Z$.  
\end{lem}

\begin{proof}

We shall first show that there is  a hyperbolic surface $Y_0$ and a minimal geodesic lamination $\lambda_0$  which has a complementary region with metric completion of the same topological type as $Z$, namely a crowned hyperbolic surface $Z_0$ of genus $g$ with a crown having $m$ boundary cusps.

For this, we construct a  bi-recurrent train-track on a surface $S$, such that one of the complementary regions has genus $g$ and $m$ corners.  

One way to do this is to start with  a train-track on a  punctured sphere with a complementary region $W$ a punctured disk with $m$ corners  (see Figure 7 --  if $m\geq 2$ the complementary region $W$  is the one containing the puncture at infinity, and if $m=1$ it is the one containing any of the other punctures). We could then replace the punctures with handles or disks to obtain a train-track on a closed surface $S$.  In particular, we replace the puncture in $W$ with $g$ handles (or a disk, in the case $g=0$) to obtain a complementary region with the topology we want. It is easy to check that this train-track is bi-recurrent.  

Then let $Y_0$ be the hyperbolic surface obtained by equipping $S$ with some hyperbolic metric. By Lemma \ref{ttlem} we can choose  weights on the train-track such that the corresponding geodesic lamination $\lambda_0$ is minimal. 

 The minimality of $\lambda_0$ ensures that the complementary region $W$ is bounded by leaves that pass through each of the branches infinitely many times. The boundary of $W$ then comprises bi-infinite geodesic lines, one for each branch, such that the lines for successive branches are asymptotic, and hence bound ``boundary cusps".

The metric completion of the complementary region corresponding to $W$ then has a crown end; denote the resulting crowned hyperbolic surface by $Z_0$.
To obtain the \textit{given} crowned  hyperbolic surface $Z$, we would now need to deform $Z_0$:
this we can achieve by taking an ideal triangulation $T_0$ of  the crowned hyperbolic surface $Z_0$, and  appropriate shearing (\textit{c.f.} Proposition \ref{wild}). 
 
Note that the ideal triangulation of $Z_0$ can be extended to an ideal triangulation of $Y_0$, and  this shearing can be performed on the hyperbolic surface $Y_0$. 
After such a shear, we obtain the required hyperbolic surface $Y$ and image of $\lambda_0$ determines the required geodesic lamination $\lambda$.
\end{proof}

\begin{figure}
\begin{center}

\includegraphics[scale=.35]{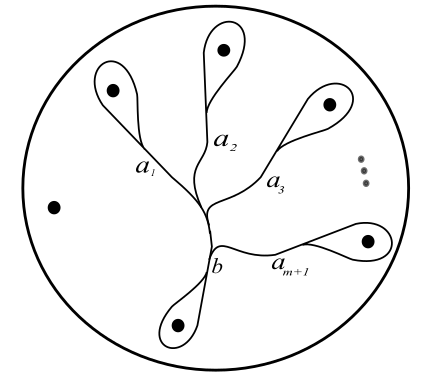}
\caption{A train-track on the punctured sphere having a complementary region that is a punctured $m$-gon. The weights satisfy $b = \sum\limits_{i=1}^{m+1}a_i$. }

\end{center}
\end{figure}

\subsubsection*{Multiple crowns}

In the case when $Z$ has $k$ crown ends, and $l$ closed geodesic boundary components,  the proof  above is suitably modified. We provide a brief description as follows. 

We now consider a train track on a punctured sphere that has $k$ components of the form in Figure 7  (with an appropriate number of corners in the complementary region), and  introduce $l$  additional components that are just loops bounding new punctures that we add to the  complementary region $W$.  Finally, we replace the puncture of the new complementary region, with $g$ handles (or a disk), to obtain a complementary region of the train-track we built, with the correct topology. The other punctures are also replaced by handles to get a closed surface $S$.

As before, we then equip $S$ with a hyperbolic metric to obtain a closed hyperbolic surface $Y_0$, and assign weights to the train-track branches such that the corresponding lamination $\lambda_0$ comprises $k$ minimal components, and $l$ components that are simple closed geodesics. 

The metric completion of $W$ would then be a crowned hyperbolic surface $Z_0$ with $k$ crown ends and  $l$ closed geodesic boundaries. We can then shear along a set of geodesic lines that divide $Z_0$ into ideal triangles, to obtain the given crowned hyperbolic surface $Z$.  This ideal triangulation can be extended to one of $Y_0$, and the shearing determines a new hyperbolic surface $Y$ and a new lamination $\lambda$.  By construction, the metric completion of $Y\setminus \lambda$ is $Z$.

\subsection{Construction of $X_t$}

In the previous section we constructed a closed hyperbolic surface $Y$ equipped with a measured geodesic lamination $\lambda$.  Let the genus of $Y$ be $g^\prime \geq g$.

It remains to construct a Teichm\"{u}ller ray $X_t$  in $\mathcal{T}_{g^\prime}$   such that its conformal limit (Definition \ref{clim}) has a component conformally homeomorphic to the given punctured Riemann surface $X \setminus P$. Moreover, the ray is to be determined by a quadratic differential $\Phi_0$ that has horizontal foliation $\phi_0^h$ that is measure-equivalent to $\lambda$. 

We choose a punctured Riemann surface $X_\infty$  homeomorphic to $Y\setminus \lambda$, one of whose components is the punctured Riemann surface $X\setminus P$.  In what follows we shall fix such a homeomorphism, such that the ends of the punctured surface  $X_\infty$ shall correspond to crowns or boundary collar of $\widehat{Y\setminus \lambda}$. In particular, $X\setminus P$ is homeomorphic to the component that is the crowned hyperbolic surface $Z$.\\

The following construction of the \Te ray with conformal limit $X_\infty$ is adapted from a similar construction in \S6.4 of \cite{Gup25}. We refer to Definition \ref{pend} for some of terminology we use.\\

\textit{Step 1.} First, we equip $X_\infty$ with a ``generalized half-plane structure", that is,  there is an infinite-area singular-flat surface  $\Sigma$ obtained by attaching Euclidean half-planes  and half-infinite cylinders to a metric graph $G$ by an isometric identification of  boundary-intervals with edges in the graph $G$, such that the resulting punctured Riemann surface is conformally equivalent to $X_\infty$. 

Suppose that a component of $\widehat{Y\setminus \lambda}$ (say $Z$) has $k$ crown ends with $m_1,m_2,\ldots m_k$ boundary cusps, and  $l$ geodesic boundary components. 
Then we require that the corresponding component of $\Sigma$ has  $k$ planar ends involving $m_1,m_2,\ldots m_k$ half-planes, and $l$ cylindrical ends, which have metric residues equal to  those of the corresponding crowns and geodesic boundaries. See Theorem 7.7 of \cite{Gup25}, and also \cite{Gup2}. \\

\textit{Step 2.} Next, for any positive real parameter $t \gg 0$, we shall describe how to  truncate the planar ends of the components of $\Sigma$ along polygons (with alternating horizontal and vertical edges) and the cylindrical ends along geodesic circles.

We start with a truncation of the crown ends of  $\widehat{Y \setminus \lambda}$ along horocylic arcs in the boundary cusps, such that for any truncated crown end, the remaining segments on the geodesic sides have hyperbolic length $t$ except one of length $t+C$, where  $C$ is the metric residue of that crown. 

This truncation determines a subsurface of $Y\setminus \lambda$, the complement of which is a ``train-track" neighborhood  $\mathcal{N}$ of the geodesic lamination $\lambda$.  In particular, the neighborhood determines a train-track  graph $T_t$, which is a ``splitting" of $T$  (see Chapter 2 of \cite{PenHar}).

Recall that a \textit{truncation} of a planar end of $\Sigma$ is obtained by removing the complement of polygonal regions on each half-plane of that planar end (see Definition \ref{pend}).

\begin{figure}
\begin{center}

\includegraphics[scale=.35]{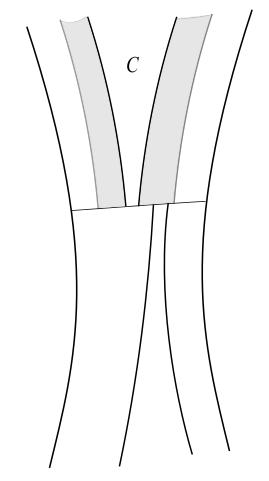}
\caption{Part of a train-track $T_t$ on $Y$. The half-branches adjacent to a truncation of a crown end $C$ determine the truncation of the corresponding planar-end of $\Sigma$. }

\end{center}
\end{figure}

The train-track $T_t$ then determines a truncation of each planar end of $\Sigma$: namely, consider the corresponding truncated crown end of $\widehat{Y \setminus \lambda}$ we constructed above, and consider the branches of $T_t$ adjacent to it. For each branch $b$ of hyperbolic length $l$ and weight $w$ (equal to the transverse measure of the measured lamination $\lambda$ across that branch) we remove the complement of a rectangular region of horizontal length $l$ and width $(t/2) 
\cdot w$ from the corresponding half-plane of the planar end. Note that this determines a truncation of the planar end since the metric residue $C$ of the crown end matches that of the planar end. See Figure 8.  Here is a brief justification of these choices:  $t\cdot w$ is the transverse measure of the same branch for the measured lamination $t\lambda$, and each branch is divided into two since it is adjacent to two (truncated) crown ends in the complement of $\lambda$.\\

\textit{Step 3.}  Finally, note that the train-track $T_t$  (corresponding to the measured lamination $t\lambda$) and its branches also determine how to assemble the singular-flat surfaces with  truncated planar ends into a closed singular-flat surface $\Sigma_t$. The horizontal foliation of the Euclidean rectangles define a foliation on the resulting singular-flat surface that exactly corresponds to the measured foliation determined by the train-track, and is hence measure-equivalent to $t \lambda$.

As $t\to \infty$, the resulting singular-flat surfaces $\Sigma_t$ differ by a stretch in the vertical direction, and hence  the underlying Riemann surfaces $X_t$ lie along a common \Te ray (\textit{c.f.} Lemma 6.12 of \cite{Gup25}) in a direction determined by $\lambda$. 
Note that in \textit{Step 2}, the lengths of the branches of $T_t$ tend to infinity as $t\to \infty$, and hence so does the sum of the horizontal edge-lengths of the polygonal boundary of the truncation of $\Sigma$, in each half-plane. Also by construction, the widths of the branches (and hence the vertical edge-lengths of the polygonal boundary) also tend to infinity linearly with $t$. Thus, in each half-plane of the planar end, the polygonal boundary encloses half-disks of radius tending to infinity as $t\to \infty$. In other words, the truncations of the planar ends of $\Sigma$ (that are isometrically embedded in $\Sigma_t$) exhaust $\Sigma$.  By Definition \ref{clim} (see also the remark following it), the conformal limit for the \Te ray $\{X_t\}_{t\gg0}$  is precisely $\Sigma$.  

In particular, the punctured Riemann surface underlying the conformal limit is $X_\infty$, that we started with. Recall that one of the components of $X_\infty$ is the punctured Riemann surface $X\setminus P$. \\

 As explained in the beginning of the section, this completes the constructions (A) and (B) and completes the proof of Theorem \ref{thm2}.

\bibliographystyle{amsalpha}
\bibliography{hrefs}

\vspace{.2in}

\end{document}